\documentclass{article}
\usepackage{amsmath,amsthm,amssymb}
\usepackage{paralist}
\usepackage{tikz}
\usepackage{graphics}
\usepackage{epsfig}
\usepackage{graphicx}
\usepackage{epstopdf}

\DeclareMathOperator{\diam}{diam}
\newcommand{\MSC}{\noindent \textit{2010 MSC:} }
\newcommand{\email}{\noindent \textit{email:} }
\newcommand{\keywords}{\noindent \textit{Keywords:} }
\newtheorem{theorem}{Theorem}[section]
\newtheorem{corollary}{Corollary}

\theoremstyle{definition}
\newtheorem{definition}[theorem]{Definition}
\newtheorem{problem}[theorem]{Problem}

\newtheorem{example}{Example}

\title{\textbf{BIFURCATIONS, ROBUSTNESS AND SHAPE OF ATTRACTORS OF DISCRETE DYNAMICAL SYSTEMS}}
\author{H\'ECTOR BARGE, ANTONIO GIRALDO\\ and\\ JOS\'E M.R. SANJURJO}
\date{}

\begin{document}
\maketitle



\begin{abstract}
We study in this paper global properties, mainly of topological
nature, of attractors of discrete dynamical systems. We consider the Andronov-Hopf bifurcation for homeomorphisms of the plane and establish some  robustness properties for attractors of such homeomorphisms. We also give relations between attractors of flows and quasi-attractors of homeomorphisms in $\mathbb{R}^{n}$. Finally, we give a result on the shape (in the sense of Borsuk) of invariant sets of IFS's on the plane, and make some remarks about the recent theory of Conley attractors for IFS.
\end{abstract}

\MSC{37C70,37G35,37B25,54C56,55P55,28A80}
\newline

\bigskip

The aim of this paper is to study global properties, mainly of topological
nature, of attractors of discrete dynamical systems and the persistence or change of these properties under perturbation or bifurcation of the system.

In the case of attractors of flows, \v{C}ech homology and Borsuk's homotopy theory (or shape theory) have been classically used in the study of the topology of attractors, see \cite{gus, san94, kar}. This study depends, in an essential way, on the homotopies provided by the flow.

However, in the discrete case, the absence of such homotopies makes the study considerably more complicated.
A way to circumvent this difficulty, at least for systems in %
 the plane, is to use a classical theorem by K. Borsuk \cite{bor}  where the shape classification of compacta is given in elementary topological terms.
For this reason, virtually no knowledge of shape theory is required in this
paper as long as we take for granted Borsuk's theorem.

In the first section, we recall some notions and results on dynamical systems, fractal geometry and topology (in particular, shape theory) needed throughout the paper.

In Section 2 we study the Andronov-Hopf bifurcation for homeomorphisms of the plane and robustness properties for attractors of such homeomorphisms. The higher dimensional case and other kind of bifurcations will be treated in a forthcoming paper.

Section 3 is dedicated to study relations between attractors of flows and attractors of homeomorphisms in $\mathbb{R}^{n}$. Here, Conley's notion of quasi-attractor plays a substantial role.

Finally, in Section 4, we study the \v{C}ech homology and shape of fractals in the plane and make some remarks about the recent theory of Conley attractors for IFS.

\section{Preliminaries}

This paper lies in the intersection of three great areas: Dynamical systems, fractal geometry and topology (in particular, shape theory). In this section we recall some basic notions of results used in the paper. For more information we refer the reader to the references cited throughout the section.

We will make use of some classical concepts of dynamical systems which can be found in \cite{bhs,rob04}. In particular, we will be dealing with attractors of flows and homeomorphisms. A compact invariant set $K$ of a flow (resp. homeomorhism) is said to be an \emph{attractor} if it admits a \emph{trapping region}, i.e., a compact neighborhood $N$ such that $Nt\subset \mathring{N}$ for every $t>0$ (resp. $f(N)\subset \mathring{N}$), satisfying that
$$
K=\bigcap_{t\geq 0}N[t,+\infty)
\ \
(\text{resp. } K=\bigcap_{k=0}^{\infty } f^k (N)).$$


It is well-known that trapping regions are robust objects, in the sense that, if $N$ is a trapping region of a flow (resp. homeomorphism) it is also a trapping region for small perturbations of it \cite{con,ken,hur}.
Finally, we will consider the more general notion of a \emph{quasi-attractor}, which is a non-empty compact invariant set which is intersection of attractors.

The main notions and ideas from fractal geometry needed in the paper can be found in \cite{bar}.
Specifically, we will study iterated function systems defined on Euclidean spaces. An \emph{iterated function system}, or shortly IFS, is a family
$$\mathcal{F}=\{f_1,f_2,\ldots,f_k\}$$
of self-maps $f_i:\mathbb{R}^n\to\mathbb{R}^n$. We will assume that these self-maps are contractive homeomorphisms. Recall that a map $f:\mathbb{R}^n\to\mathbb{R}^n$ is \emph{contractive} if there exists $\lambda \in [0,1)$ such that
$d(f(x),f(y))\leq \lambda d(x,y)$.

If $\mathcal{F}$ is an IFS consisting of contractive maps it has a unique compact invariant set $K$, i.e.
\[
K=\bigcup_{i=1}^kf_i(K).
\]
Besides, $K$ attracts every non-empty compact subset of $\mathbb{R}^n$ in the Hausdorff metric, defined in the hyperspace
$
\mathcal{H}(\mathbb{R}^n)=\{K\subset\mathbb{R}^n\mid K\;\mbox{is non-empty and compact}\}
$
as
\[
d_H(A,B)=\inf\{\epsilon\geq 0\mid B\subset A_\epsilon,\;\mbox{and}\; A\subset B_\epsilon\}
\]
where
$A_\epsilon=\bigcup_{x\in A}\{y\in\mathbb{R}^n\mid d(x,y)\leq\epsilon\} $.
Hence, $K$ is known as the attractor of the IFS.

Finally, we will use some topological notions in the paper, in particular notions and results from shape theory, a form of homotopy theory introduced and studied by K. Borsuk \cite{bor}, which has been proved to be specially  convenient for the study of the global topological properties of the invariant spaces involved in dynamics. In particular, we will make use of
the following result which characterizes the shape of all plane continua in terms of the number of components of their complements.

\begin{theorem}[\cite{bor}] Two continua $K$ and $L$ contained in $\mathbb{R}^{2}$
have the same shape if and only if they disconnect $\mathbb{R}^{2}$ in the
same number (finite or infinite) of connected components.
In particular:
\begin{enumerate}[i)]
\item  A continuum has trivial shape (the shape of a point) if and only if it does not disconnect $\mathbb{R}^{2}$.
\item A continuum has the shape of a circle if and only if it
disconnects $\mathbb{R}^{2}$ into exactly two connected components.
\item Every other continuum has the
shape of a wedge of circles, finite, or infinite (Hawaiian earring).
\end{enumerate}
\end{theorem}

We will denote by $\check{H}_*$ and $\check{H}^*$ the \v{C}ech homology and cohomology functors respectively. \v{C}ech homology (cohomology) theory has proven to be important in dynamical systems since, for the case of attractors, it agrees with the homology of its basin of attraction in the case of flows defined on manifolds (more generally ANR's) \cite{gus,san11}, and in the case of homeomorphisms in
manifolds, if we take coefficients in a field \cite{rps}.
In both cases it must be finitely generated. Moreover, \v Cech homology and cohomology are shape invariants.

We also recall that a compact subset $K$ of $\mathbb{R}^n$ is said to be \emph{cellular} if it has a neighborhood basis consisting of topological balls. We will make use of the fact that, if $K$ is a cellular subset of $\mathbb{R}^n$, then $\mathbb{R}^n\setminus K$ is homeomorphic to $\mathbb{R}^n\setminus \{p\}$, where $p\in\mathbb{R}^n $ \cite{dav}.
In particular $\mathbb{R}^n\setminus K$ is connected. Cellular sets are instances of continua with trivial shape.

For general information on algebraic topology we recommend the book by {Spa\-nier} \cite{spa}.
For a complete treatment of shape theory we refer the reader to \cite{bor,cop,dys,mar,mas}.
The use of shape in dynamics is illustrated by the papers
\cite{gmrs,gis,gus,has,kar,ros,rob99,sag}.

\section{Bifurcations and robustness of attractors of {ho\-meo\-mor\-phisms} of the plane}

In this section we prove some results related to the Andronov-Hopf
bifurcation theorem for diffeomorphisms of the plane. This result was proved
by Naimark \cite{nai}, Sacker \cite{sac} and Ruelle and Takens \cite{rut}. The name of Andronov-Hopf bifurcation is commonly used because
of the connection with the Andronov-Hopf bifurcation for differential
equations. The Andronov-Hopf bifurcation for a diffeomorphism occurs when a
pair of eigenvalues for a fixed point changes from absolute value less than
one to absolute value greater than one, i.e., the fixed point changes from
stable to unstable by a pair of eigenvalues crossing the unit circle. The
formulation of the theorem and its proof are rather complicated. In the case
of flows there are some topological versions of the theorem where the
hypotheses are simpler and the conclusions are weaker but significant, see \cite{san07}.
To reach these conclusions, a heavy use is made of the homotopies induced by
the flow. In the case of discrete dynamical systems induced by
homeomorphisms, we don't have such homotopies at our disposal. However, by
using Borsuk's theorem on the shape classification of plane continua \cite{bor}, we are able to prove a result which can be considered as a
topological Andronov-Hopf bifurcation theorem for homeomorphisms of the
plane. We also prove in this section a result concerning robustness of some
global properties of attractors of homeomorphisms of the plane (formulated in
the language of shape theory and \v{C}ech homology). This result was already
known for flows (see \cite{san95}).
However, as in the previous result, we cannot use the flow-induced homotopies which are
essential in the proof of the continuous case.
We use, instead, Borsuk's theorem again to give a discrete counterpart of that result.

\begin{theorem}
Suppose that $\mathbf{\Phi }=(\Phi _{\lambda })_{\lambda \in \mathbb{R}}:\mathbb{R}^{2}\times \mathbb{R\rightarrow R}^{2}$ is a continuous
one-parameter family of homeomorphisms of the plane such that
\begin{enumerate}[(a)]
\item the origin is a fixed point of $\Phi _{\lambda }$ for $\lambda $ near $0$,
\item the origin is an attractor for $\lambda =0$,
\item the origin is a repeller for $\lambda >0$.
\end{enumerate}

Then there exists a $\lambda _{0}>0$ such that for every $\lambda $ with
$0<\lambda \leq \lambda _{0}$ there exists an attractor $K_{\lambda }$ of
$\Phi _{\lambda }$ with $Sh(K_{\lambda })=Sh(\mathbb{S}^{1})$. In particular,
the \v{C}ech homology and cohomology of $K_{\lambda }$ agree with that of
the circle. Moreover, the attractors $K_{\lambda }$ surround the fixed point
$\{0\}$ and shrink to it when $\lambda \rightarrow 0$.
\end{theorem}

\begin{proof}
Since $\{0\}$ is an attractor for $\Phi _{0}$, there exists a trapping
region, i.e. a compact neighborhood $N$ of $\{0\}$ such that $\Phi
_{0}(N)\subset \mathring{N}$, satisfying that $\{0\}=\bigcap_{n=0}^{\infty } \Phi _{0}^n (N)$.
Select a closed disk $D$ centered at $\{0\}$
and contained in $\mathring{N}$ and a $k>0$ such that $\Phi
_{0}^{k}(N)\subset \mathring{D}$. By the continuity of the family $\Phi $
there is a $\lambda _{0}$ such that $N$ is also a trapping region for $\Phi
_{\lambda }$ and $\Phi _{\lambda }^{k}(N)\subset \mathring{D}$ (and hence
$\Phi _{\lambda }^{k}(D)\subset \mathring{D}$) for every $\lambda $ with
$\lambda \leq \lambda _{0}$. Then for every $\lambda \leq \lambda _{0}$ there
exists an attractor $A_{\lambda }$ of $\Phi _{\lambda }$ with $A_{\lambda
}=\bigcap _{n=0}^{\infty }\Phi _{\lambda }^{n}(N)$. Since $\Phi _{\lambda
}^{k}(N)\subset \mathring{D}$ we have that $\Phi _{\lambda }^{(n+1)k}(N)$ is
contained in $\Phi _{\lambda }^{nk}(\mathring{D})$ for every $n\geq 1$ and,
thus, $\Phi _{\lambda }^{(n+1)k}(D)$ is also contained in $\Phi _{\lambda
}^{nk}(\mathring{D})$ (which agrees with the interior of $\Phi _{\lambda
}^{nk}(D)$) and
$$
A_{\lambda }=\bigcap_{n=1}^{\infty } \Phi _{\lambda }^{(n+1)k}(N)\subset \bigcap_{n=1}^{\infty }  \Phi _{\lambda }^{nk}(D)\subset \bigcap_{n=1}^{\infty }  \Phi _{\lambda }^{nk}(N)=A_{\lambda }.
$$

Hence $A_{\lambda }$ is the intersection of the topological disks $\Phi
_{\lambda }^{nk}(D)$, and the fact that $\Phi _{\lambda }^{(n+1)k}(D)$ is
contained in the interior of $\Phi _{\lambda }^{kn}(D)$ for every $n$
implies that $A_{\lambda }$ is a cellular set. Although we will not need it,
we point out that an alternative description of $A_{\lambda }$ is
$$
A_{\lambda }=\{y\in \mathbb{R}^{2}|\text{ there exist sequences }x_{n}\in D,
\text{ }k_{n}\rightarrow \infty \text{ with }\Phi _{\lambda
}^{k_{n}}(x_{n})\rightarrow y\}.
$$

Then we have for every $\lambda \leq \lambda _{0}$ a cellular attractor
$A_{\lambda }\subset \mathring{D}$ of $\Phi _{\lambda }$ such that $D$ is
contained in its basin of attraction, which we denote by $\mathcal{A} _{\lambda }$.
We remark that $\{0\}$ is a fixed point of $\Phi _{\lambda }$
contained in $A_{\lambda }$. Moreover, since $\{0\}$ is a repeller for $\Phi
_{\lambda }$ we must also have that its basin of repulsion $\mathcal{R} _{\lambda }$ is contained in $A_{\lambda }$. Notice that $\mathcal{R} _{\lambda }$ is open and connected.
Now we define $K_{\lambda }=A_{\lambda }\setminus \mathcal{R}_{\lambda }$. Then $K_{\lambda }$ is a connected attractor of $\Phi _{\lambda }$ whose basin of attraction is
$\mathcal{A}_{\lambda }\setminus \{0\}$.
Since $\mathbb{R}^{2}\setminus K_{\lambda }=(\mathbb{R}^{2}\setminus A_{\lambda })\cup
\mathcal{R}_{\lambda }$ we have that $K_{\lambda }$ decomposes the plane
into two connected components and, by Borsuk's theorem on shape
classification of plane continua \cite{bor}, that $Sh(K_{\lambda })=Sh(\mathbb{S} ^{1})$.
Moreover $K_{\lambda }$ surrounds $\{0\}$ because the origin is in
$\mathcal{R}_{\lambda }$ which is the bounded component of $\mathbb{R} ^{2}\setminus K_{\lambda }$.
We remark that the disk $D$ could have been chosen
arbitrarily small and that $A_{\lambda }\subset D$ for $\lambda $
sufficiently small. Hence the attractors $A_{\lambda }$ (and consequently
also the attractors $K_{\lambda }$) converge to the origin.
\end{proof}

The following result states that attractors of homeomorphisms of the plane
are robust, in the sense that some of their global properties are preserved
under small perturbations of the homeomorphism.

\begin{theorem}
Suppose that $\mathbf{\Phi }=(\Phi _{\lambda })_{\lambda \in \mathbb{R}}: \mathbb{R}^{2}\times \mathbb{R\rightarrow R}^{2}$ is a continuous
one-parameter family of homeomorphisms and that $A$ is an attractor of $\Phi
_{0}$. Then there exists $\lambda _{0}>0$ such that for every $\lambda \leq
\lambda _{0}$ there exists an attractor $A_{\lambda }$ of $\Phi _{\lambda }$
with $Sh(A_{\lambda })=Sh(A)$. In particular $K_{\lambda }$ has the same
\v{C}ech homology and cohomology as $K$. Moreover $A_{\lambda }\rightarrow A$
when $\lambda \rightarrow 0$ (i.e. all $A_{\lambda }$ are contained in an
arbitrary neighborhood of $A$ in $\mathbb{R}^{2}$ for $\lambda $
sufficiently small).
\end{theorem}

\begin{proof}
An argument similar to the one used in the proof of Theorem 2.1 shows that if $N$ is a trapping region for the attractor $A$ of $\Phi _{0}$ then $N$ is
also a trapping region for every $\Phi _{\lambda }$ with $\lambda $
sufficiently small. Hence there exists an attractor $A_{\lambda }$ for $\Phi
_{\lambda }$ corresponding to the trapping region $N$.

On the other hand, by Corollary 2 below (see also \cite{rps}), $\mathbb{R}^{2}\setminus A$ has a finite number of connected components (say $r$).
Then there exists a topological disk with $r-1$ holes $D\subset N$ containing $A$
in its interior and such that the inclusion $j:A\hookrightarrow D$ is a shape
equivalence. Moreover, as in the proof of Theorem 2.1, there exists a $k>0$
such that $A$ is the intersection of the sets $\Phi ^{(n+1)k}(D)\subset \Phi
^{nk}(\mathring{D})$ with $n=0,1,\dots $. Since $j:A\hookrightarrow D$ is a shape
equivalence and $\Phi ^{k}(D)$ is also a topological disk with $r-1$ holes
contained in $D$ and containing $A$, then the inclusion $\Phi
^{k}(D)\hookrightarrow D$ is a homotopy equivalence.
Now, as in Theorem 2.1, $A_{\lambda }$ is the intersection of the sets $\Phi _{\lambda
}^{(n+1)k}(D)\subset \Phi _{\lambda }^{nk}(\mathring{D})$ with $n=0,1,\dots $.
If we choose $\lambda $ sufficiently small, then the inclusion $\Phi
_{\lambda }^{k}(D)\hookrightarrow D$ is also a homotopy equivalence and, thus,
the inclusion $\Phi _{\lambda }^{(n+1)k}(D)\hookrightarrow \Phi _{\lambda
}^{nk}(D)$ is a homotopy equivalence for every $n$ as well. Hence, since
$K_{\lambda }=\bigcap _{n=0}^{\infty }\Phi _{\lambda }^{nk}(D),$ then $K_{\lambda }$ has the shape of a topological disk with $r-1$ holes and,
thus, $Sh(K_{\lambda })=Sh(K)$.
The last part of the theorem is a
consequence of the fact that $D$ can be chosen arbitrarily small.
\end{proof}

\section{Quasi-attractors of flows and attractors of {ho\-meo\-mor\-phisms} of $\mathbb{R}^{n}$}

It is well-known that there are continua in $\mathbb{R}^{n}$ which are
attractors of homeomorphisms of $\mathbb{R}^{n}$ but which are not
attractors of flows defined on $\mathbb{R}^{n}$. One of such examples is the
solenoid in $\mathbb{R}^{3}$. It is, however, interesting to explore
possible relations between these two notions. An interesting connection can
be found by using the notion of quasi-attractor of a flow. This notion,
introduced by C. Conley, plays an important role in his study of the chain
recurrent set and the gradient structure of a flow \cite{con}. There is a similar notion of quasi-attractor of a homeomorphism, which was used by J. Kennedy \cite{ken}, and M. Hurley \cite{hur} in their study of the generic properties of discrete dynamical systems.

\begin{definition}
Let $\varphi :\mathbb{R}^{n}\times \mathbb{R}\rightarrow \mathbb{R}^{n}$ be
a flow. A non-empty compactum $K\subset \mathbb{R}^{n}$ is said to be a
quasi-attractor of $\varphi $ if it is the intersection of a family of
attractors of $\varphi $. The non-empty compactum $K$ is a tame quasi-attractor of $\varphi $ if it is the intersection of a nested sequence of attractors of $\varphi $,  $A_{1}\supset A_{2}\supset \dots \supset A_{n}\supset \dots $, all of
them homeomorphic.
\end{definition}

Our next result shows that tame quasi-attractors of flows share with
attractors of flows the following important property.

\begin{theorem}
Let $K$ be a tame quasi-attractor of a flow on $\mathbb{R}^{n}$. Then the
\v{C}ech homology and cohomology of $K$ with coefficients in a field is
finitely generated in every dimension.
\end{theorem}

\begin{proof}
Suppose that $K$ is the intersection of a family of attractors
$$
K_{1}\supset K_{2}\supset \dots \supset K_{k}\supset K_{k+1}\supset \dots
$$
of a flow $\varphi :\mathbb{R}^{n}\times \mathbb{R\rightarrow R}^{n}$, all
of them homeomorphic. Every inclusion
$j_{k}:K_{k+1}\hookrightarrow K_{k}$
induces a
homomorphism $\check{H}_{r}(j_{k}):\check{H}_{r}(K_{k+1})\rightarrow
\check{H}_{r}(K_{k})$ and in this way we obtain an inverse sequence of
vector spaces
$$
\check{H}_{r}(K_{1})\leftarrow \check{H}_{r}(K_{2})\leftarrow \dots \leftarrow \check{H}_{r}(K_{k})\leftarrow
\check{H}_{r}(K_{k+1})\leftarrow \dots
$$
where coefficients are taken in a field. Then, by the continuity of \v{C}ech
homology, we have that $\check{H}_{r}(K)=\underleftarrow{\lim }\check{H} _{r}(K_{k})$. It is known that the \v{C}ech homology of an attractor is
finitely generated (indeed attractors of flows have finite polyhedral
shape), see \cite{gus, san94, kar}. Since the attractors $K_{k}$ are
homeomorphic we have that all $\check{H}_{r}(K_{k})$ are vector spaces
isomorphic to a finite-dimensional vector space $V$ (the same for every $k$)
and the previous inverse sequence takes (up to isomorphism) the form
$$
V\leftarrow V\leftarrow \dots \leftarrow V\leftarrow V\leftarrow \dots
$$
where every arrow represents an endomorphism $h_{k}:V\rightarrow V$ (which
can vary with $k$). Since $V$ \ is finite-dimensional then for every $k$
there exists a subspace $V_{k}$ of $V$ such that $V_{k}$ is the image of the
composition of homomorphisms $h_{k}\circ h_{k+1}\circ \dots \circ h_{k+l}$ \
for every $l\in \mathbb{N}$ except, perhaps, for a finite number of them.
Replacing the inverse sequence, if necessary, by a subsequence we can assume
that the image of the restriction
$h_{k}|V_{k+1}:V_{k+1}\rightarrow V$
is $h_{k}(V_{k+1})=V_{k}$. On the other hand, since $\dim (V_{k})\leq \dim (V_{k+1})$ and $V$ is finite-dimensional then necessarily dim $(V_{k})$ is
the same for almost every $k$ and, hence, $h_{k}|V_{k+1}:V_{k+1}\rightarrow
V_{k}$ is an isomorphism for almost every $k$. On the other hand, the
inverse limit of
$$
V\leftarrow V\leftarrow \dots \leftarrow V\leftarrow V\leftarrow \dots
$$
agrees with the inverse limit of
$$
V_{1}\leftarrow V_{2}\leftarrow \dots \leftarrow V_{k}\leftarrow V_{k+1}\leftarrow \dots
$$
where the arrows are the homomorphisms $h_{k}|V_{k+1}:V_{k+1}\rightarrow
V_{k}$. Since almost all these arrows are isomorphisms, the inverse limit of
the former sequence is isomorphic to $V_{k}$ for sufficiently high $k$.
Hence $\check{H}_{r}(K)$ is finitely generated. The proof for cohomology is
similar.
\end{proof}

We prove in the following result that attractors of homeomorphisms and tame
quasi-attractors of flows are intimately related.

\begin{theorem}
Every continuum of $\mathbb{R}^{n}$ is a quasi-attractor of a flow on $\mathbb{R}^{n}$. Every connected attractor of a homeomorphism of $\mathbb{R} ^{n}$ is a tame quasi-attractor of a flow on $\mathbb{R}^{n}$.
\end{theorem}

\begin{proof}
Every continuum $K$ of $\mathbb{R}^{n}$ is the intersection of a nested
sequence of compact connected $n$-manifolds with boundary $(M_{k})_{k\in
\mathbb{N}}$ with $M_{k+1}\subset \mathring{M}_{k}$. Consider for every $k\in \mathbb{N}$
a collar $N_{k}$ of $\partial M_{k}$ which is, by definition, an open
neighborhood $N_{k}$ of $\partial M_{k}$ in $M_{k}$ homeomorphic to $\partial
M_{k}\times \lbrack 0,1)$ by a homeomorphism
$$h_{k}:\partial M_{k}\times \lbrack 0,1)\rightarrow N_{k}.$$
We may assume that $\bar{N}_{k} \cap M_{k+1}=\emptyset $. We define a flow in $\partial M_{k}\times \lbrack 0,1/2]
$ such that all points in $(\partial M_{k}\times \{0\})\cup (\partial
M_{k}\times \{1/2\})$ are equilibria and the trajectories of points $(x,t)\in \partial M_{k}\times (0,1/2)$ go along $\{x\}\times (0,1/2)$
connecting $(x,0)$ to $(x,1/2)$. By using the homeomorphism $h_{k}$ we carry
this flow to a flow defined in a subset of $N_{k}$. The non-stationary
trajectories of all the flows just defined form a family $\mathcal{C}$ of
oriented curves filling an open set $U$ of $\mathbb{R}^{n}$ which is regular
in the sense of Whitney \cite{whi}. We recall that a family of oriented curves $\mathcal{C}$ is regular if given an oriented arc $pq\subset \gamma \in
\mathcal{C}$ and $\epsilon >0$ there exists $\delta >0$ such that if
$p'\in \gamma '\in \mathcal{C}$ and $d(p,p')<\delta $ then there is a point $q'\in \gamma '$ such
that the oriented arcs $pq$ and $p'q'$ have a parameter
distance less that $\epsilon $ (that is, there exist parametrizations
$f:[0,1]\rightarrow pq$ and $f':[0,1]\rightarrow p'q'$ such that $d(f(t),f'(t))<\epsilon $ for every $t\in
[ 0,1]$). Consider now the partition $\mathcal{D}$ of $\mathbb{R}^{n}$
given by $\mathcal{C}$ and the singletons corresponding to the points of $\mathbb{R}^{n}$ not lying in such trajectories. By Whitney's Theorem 27A in \cite{whi}
there exists a flow $\varphi $ in $\mathbb{R}^{n}$ whose oriented
trajectories correspond to the elements of $\mathcal{D}$. It is easy to see
that $K_{k}=M_{k}\setminus h_{k}(\partial M_{k}\times [ 0,1/2))$ is an
attractor for every $k$ and that $K=\bigcap _{k\in \mathbb{N} } K_{k}$. Moreover, if $K$ is
attractor of a homeomorphism $f$ of $\mathbb{R}^{n}$ then the manifolds $M_{k}$ can be taken as images $f^{n_{k}}(M)$ of a neighborhood $M$ of $K$, which is also a manifold, and the collars $N_{k}$ can be taken as $f^{n_{k}}(N)$ where $N$ is a collar of $\partial M$ in $M$. Then the
manifolds $M_{k}$ are homeomorphic and also are the attractors $K_{k}$.
Hence $K$ is a tame quasi-attractor.
\end{proof}

\begin{corollary}
Let $K$ be a connected attractor of a homeomorphism of $\mathbb{R}^{n}$.
Then there is a flow $\varphi $ on $\mathbb{R}^{n}$ such that $K$ is the
limit in the Hausdorff metric of a sequence of attractors of $\varphi ,$ all
of them homeomorphic.
\end{corollary}

As a consequence of Theorem 3.2 and 3.3 we obtain the following result which was proved by F. Ruiz del Portal and J.J. S\'{a}nchez-Gabites for cohomology with coefficients in $\mathbb{Q} $ and $\mathbb{Z} _p $ ($p$ prime) in \cite{rps} where cohomological properties of attractors of discrete dynamical systems were studied.

\begin{corollary}
Let $K$ be a connected attractor of a homeomorphism of $\mathbb{R}^{n}$.
Then the \v{C}ech homology and cohomology of $K$ with coefficients in a field is
finitely generated in every dimension.
\end{corollary}

\section{Fractals in the plane and Conley attractors for IFS}

It is well-known that an attractor of a flow on a manifold must have the shape of a finite polyhedron \cite{gus, san94, kar}.
This fact imposes a strong restriction on the family of plane continua which may be attractors of flows. In particular, a continuum with the
shape of the Hawaiian earring cannot be attractor of a flow in the plane.
Similarly, we can ask ourselves which are the shapes that may have the fractals in the plane.
It can be easily shown, by using the Collage Theorem \cite{behl}, that all shapes are possible among the connected attractors of
IFS of contractive similarities in the plane. The following is an example of
one of such fractals with the shape of the circle.

\begin{example}
Consider the IFS with six similarities $h_1 , h_2 ,\dots ,h_{6} $, given by:
\begin{align*}
h_1 (x,y)&=\left( \dfrac{19}{30} x,\dfrac{19}{30} y\right) &
h_2 (x,y)&=\left( \dfrac{1}{2}\, \dfrac{11}{30} +\dfrac{19}{30} x,\dfrac{\sqrt{3}}{2}\, \dfrac{11}{30} +\dfrac{19}{30} y\right) \\
h_3 (x,y)&=\left( \dfrac{11}{30} +\dfrac{19}{30} x,\dfrac{19}{30} y\right) &
h_4 (x,y)&=\left( \dfrac{1}{2} x,\dfrac{1}{2} y\right) \\
h_5 (x,y)&=\left( \dfrac{1}{4} +\dfrac{x}{2} ,\dfrac{\sqrt{3}}{4} +\dfrac{y}{2} \right) &
h_6 (x,y)&=\left( \dfrac{1}{2} +\dfrac{x}{2} , \dfrac{y}{2} \right)
\end{align*}

The invariant set of the IFS and its image under the similarities are shown in Figure 1.

\begin{figure}[ht]
\begin{center}
\includegraphics[angle=0,width=3.4cm,keepaspectratio=true]{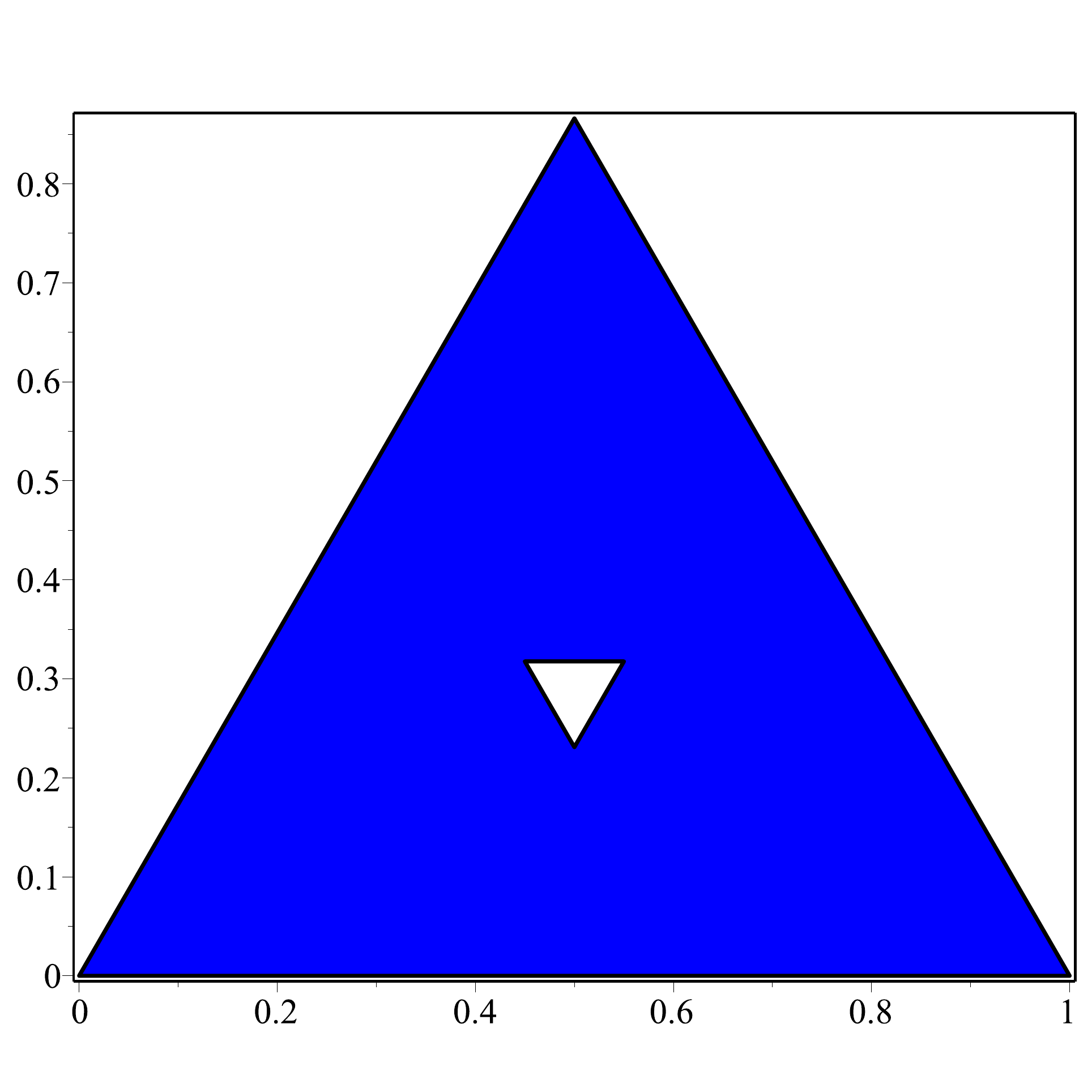}

\includegraphics[angle=0,width=3.4cm,keepaspectratio=true]{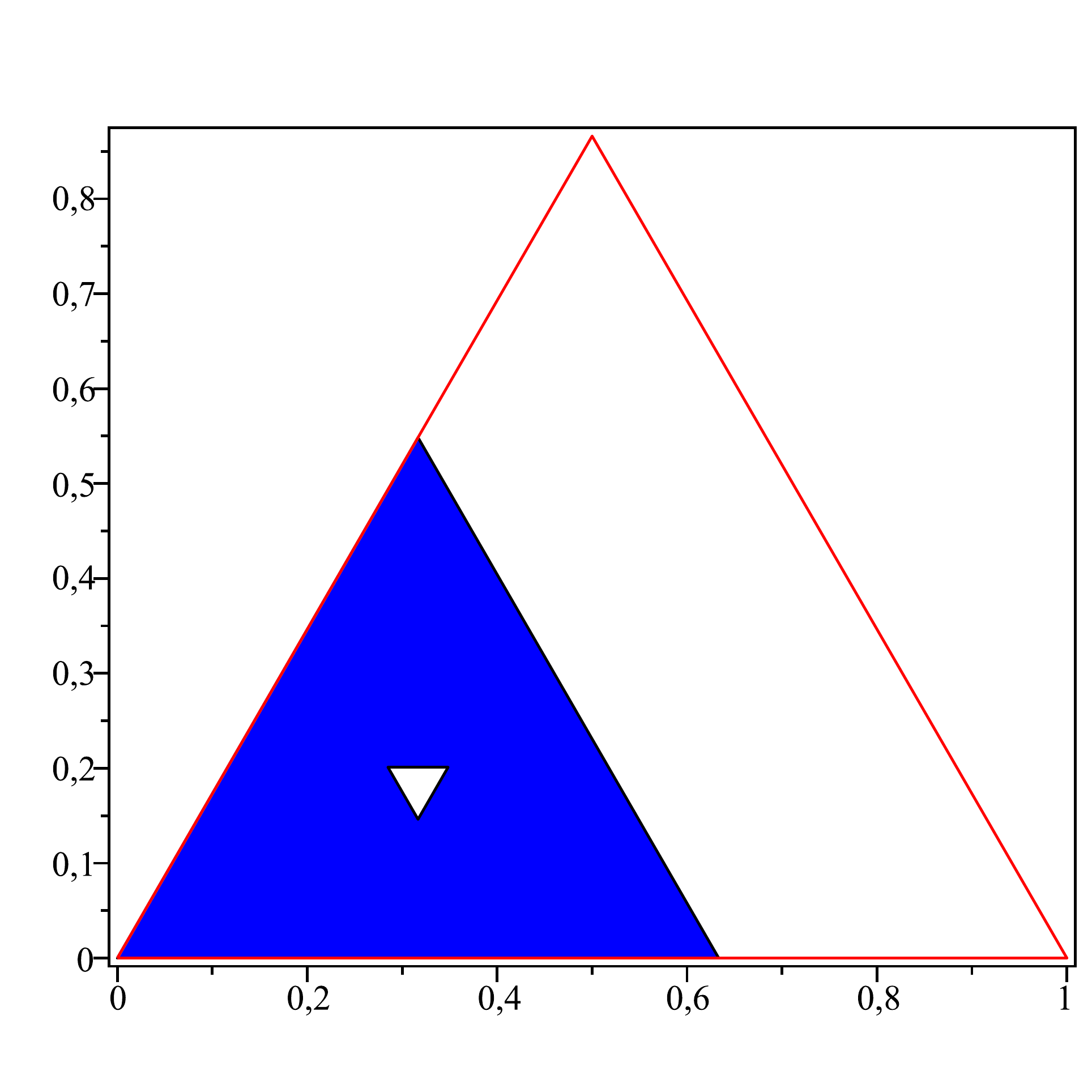}
\includegraphics[angle=0,width=3.4cm,keepaspectratio=true]{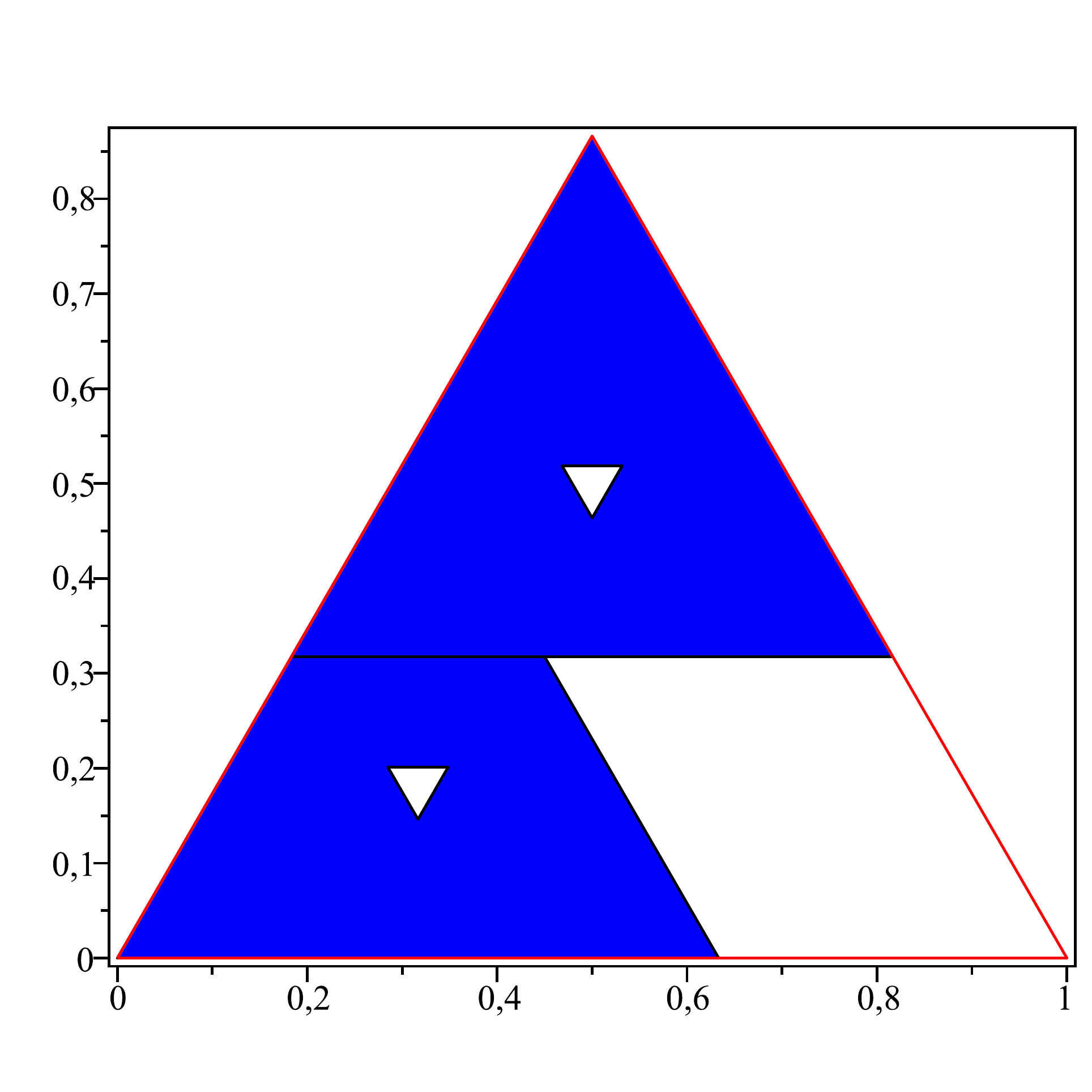}
\includegraphics[angle=0,width=3.4cm,keepaspectratio=true]{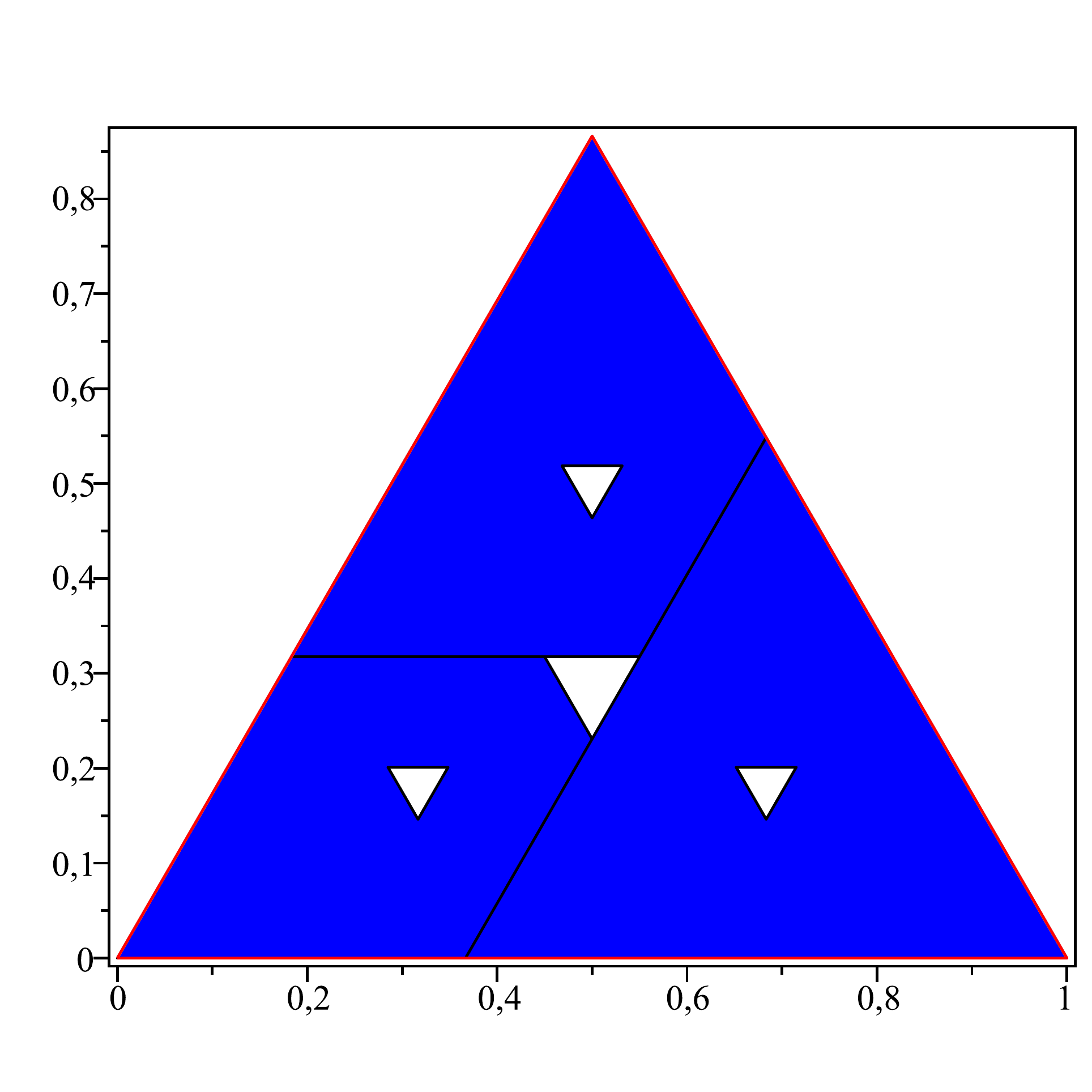}
\includegraphics[angle=0,width=3.4cm,keepaspectratio=true]{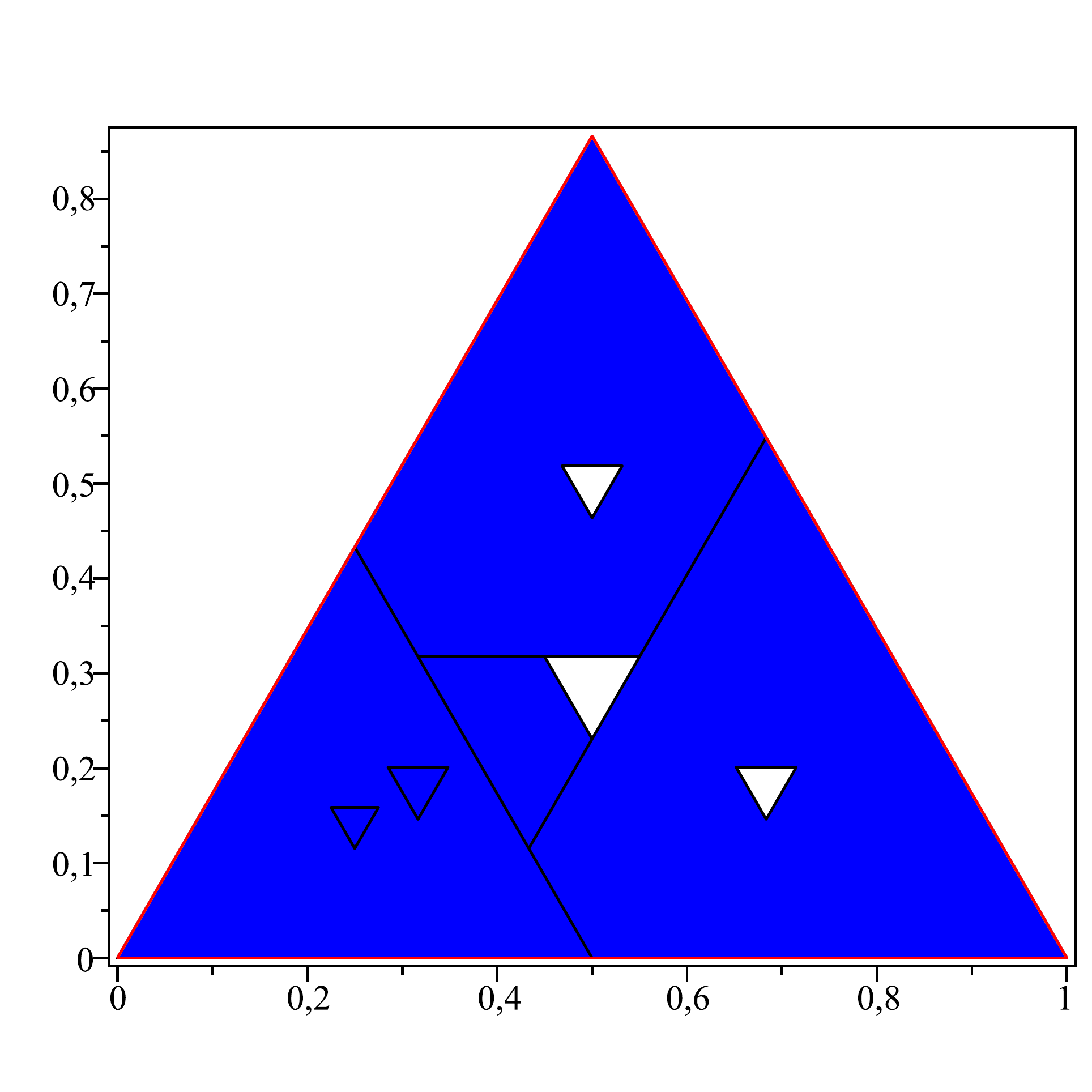}
\includegraphics[angle=0,width=3.4cm,keepaspectratio=true]{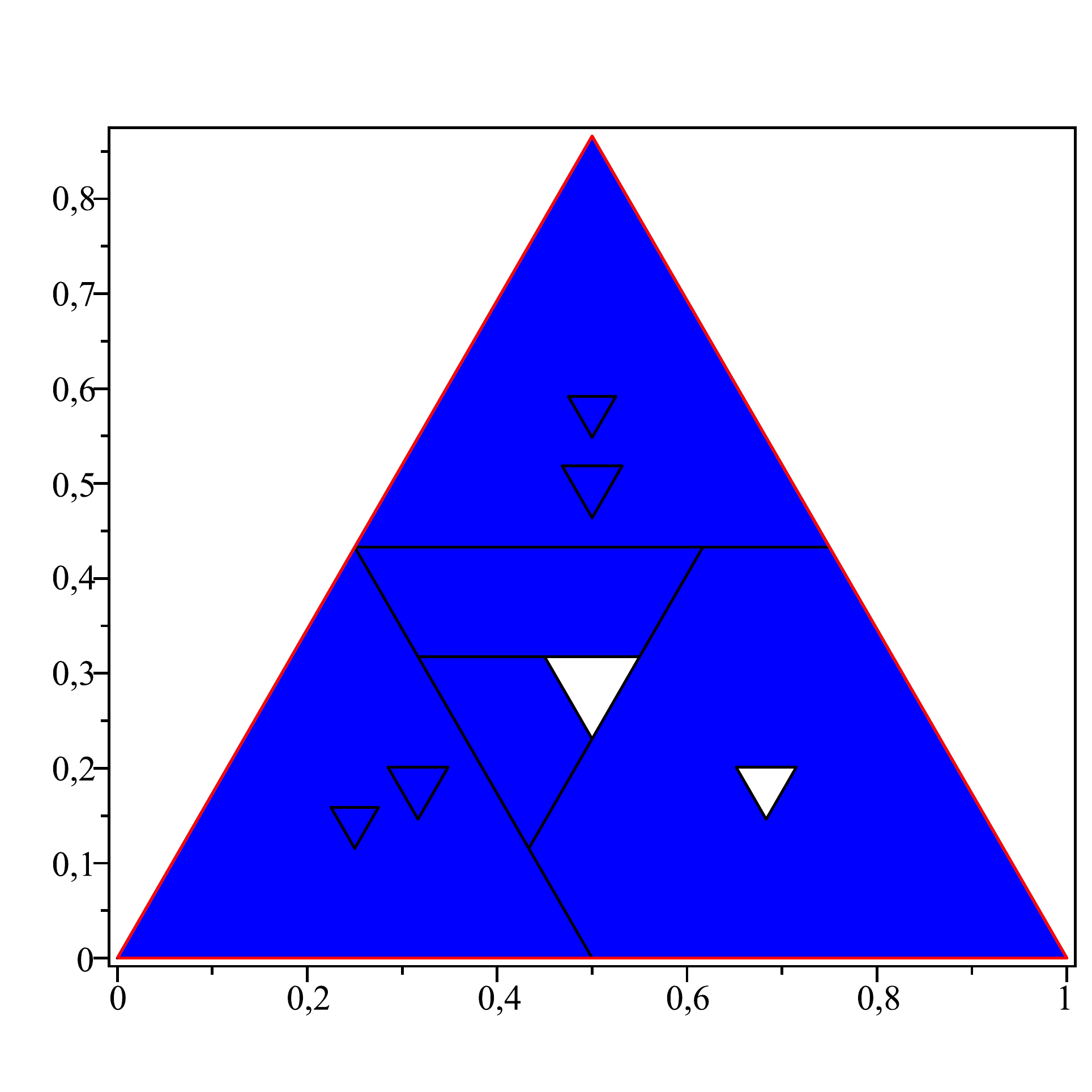}
\includegraphics[angle=0,width=3.4cm,keepaspectratio=true]{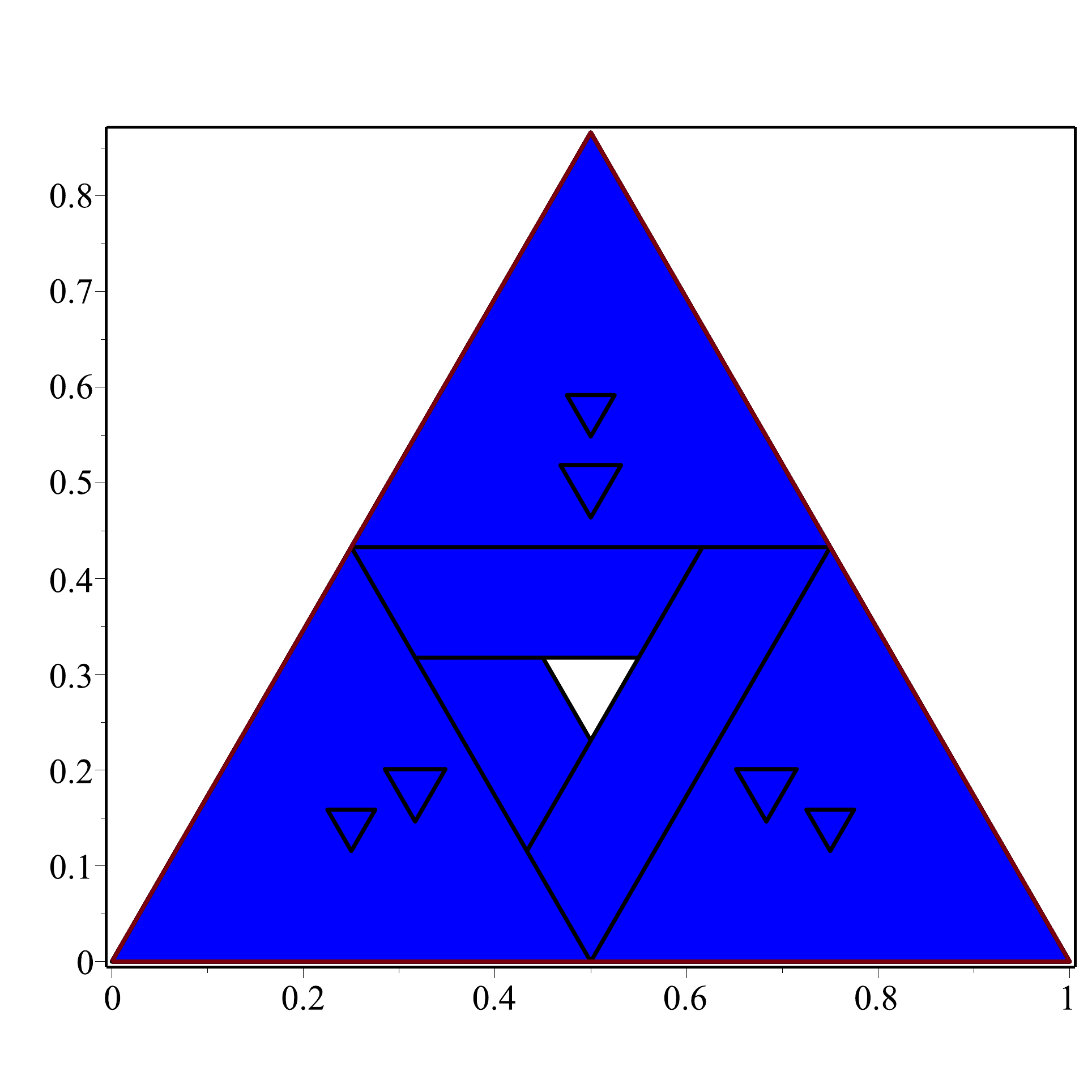}
\caption{The invariant set of the IFS and its image under the similarities}
\end{center}
\end{figure}

If we consider the IFS formed by just the last three maps, the invariant set is the classical Sierpinski gasket (Figure 2, right). On the other hand, the invariant set for the IFS formed by the first three maps is shown in Figure 2 (left).

\begin{figure}[ht]
\begin{center}
\includegraphics[angle=0,width=3.4cm,keepaspectratio=true]{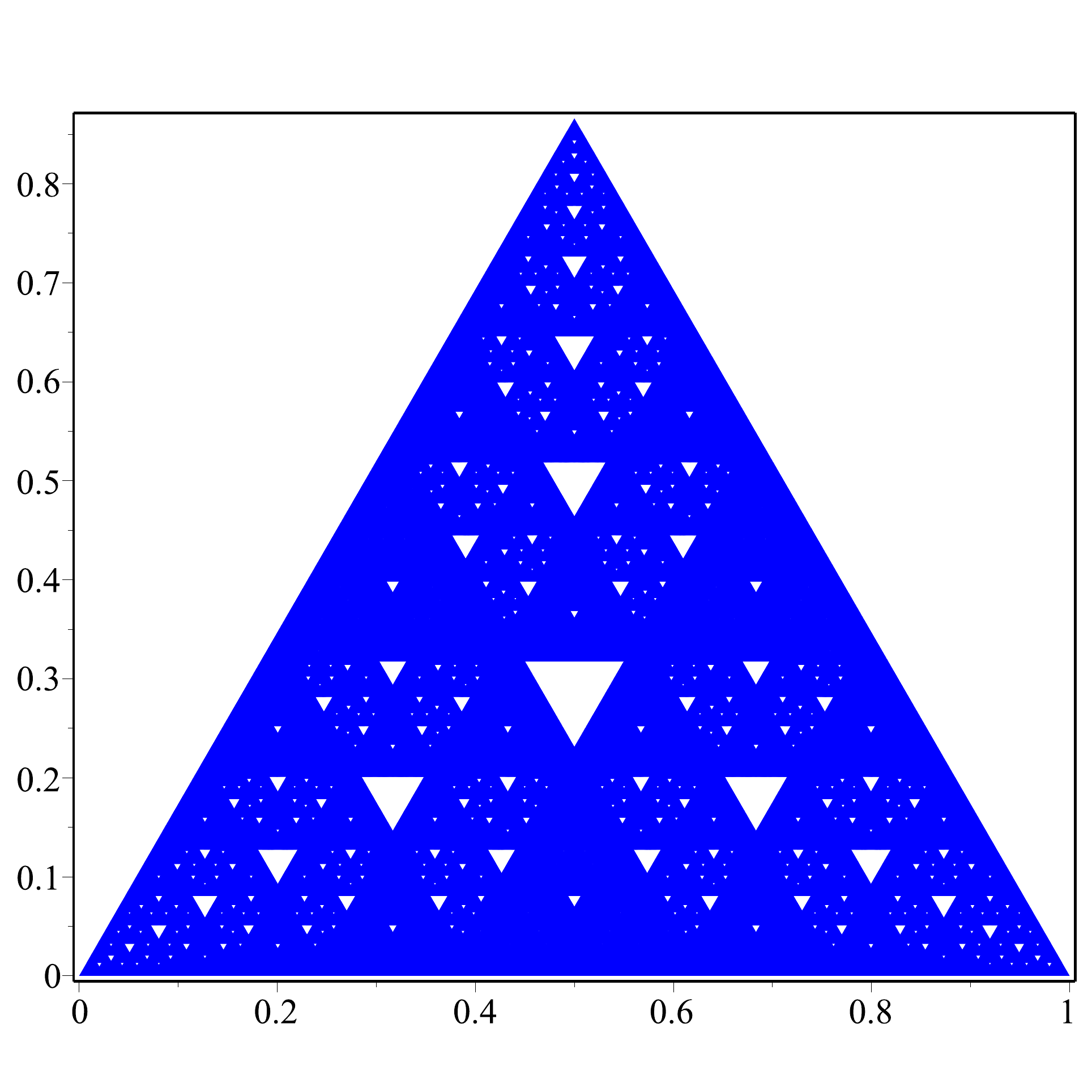}
\includegraphics[angle=0,width=3.4cm,keepaspectratio=true]{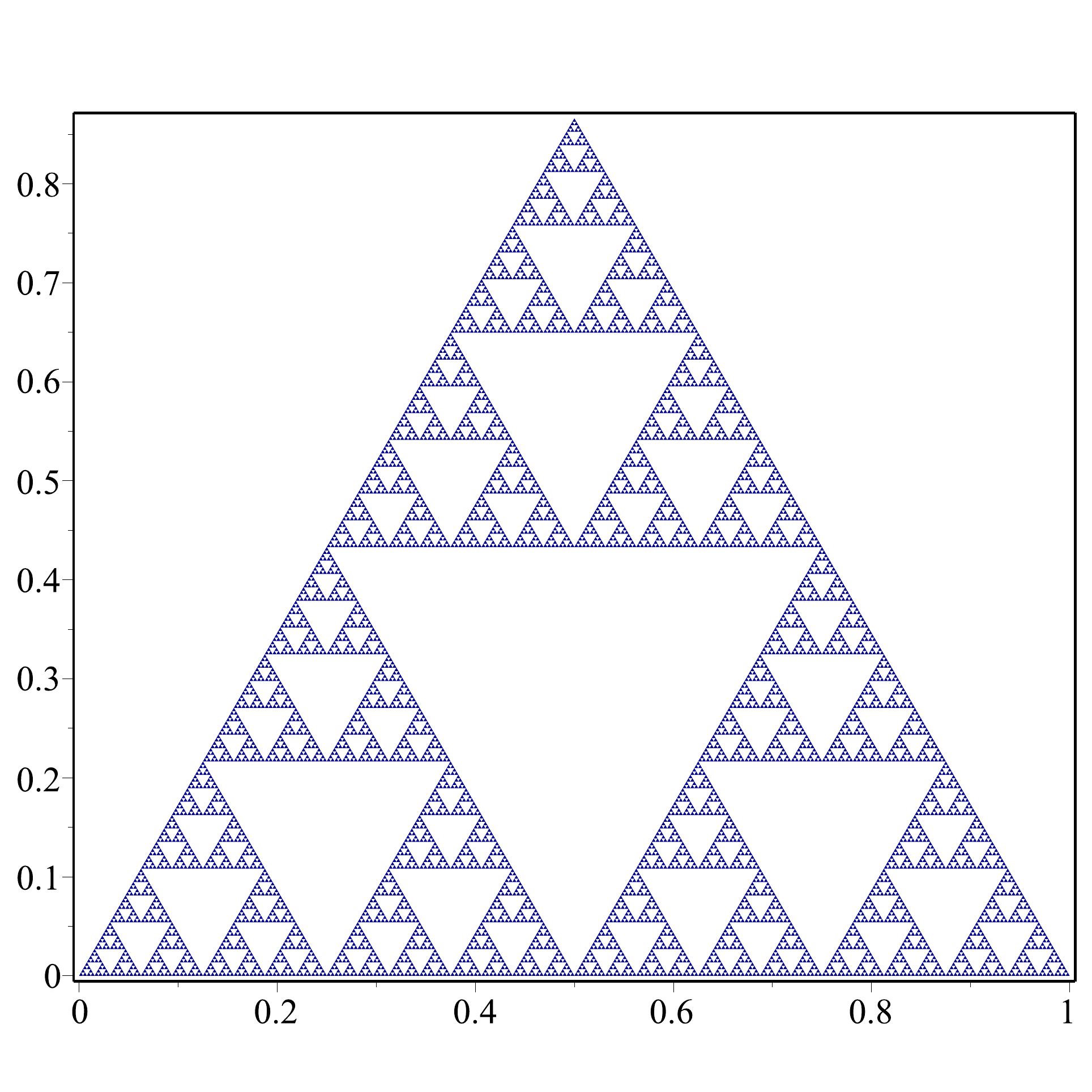}
\caption{The invariant set of the IFS formed by the first three maps (left) and the three last maps (right)}
\end{center}
\end{figure}
\end{example}

However, if we consider only continua with empty interior, all shapes are no longer possible. In fact, imposing this simple restriction, only two shapes are possible, exemplified by the Von Koch curve and the Sierpinski triangle.
These are two extreme cases among the shapes in the plane, corresponding to trivial shape and the shape of the Hawaiian earring, respectively.
The intermediate shapes of finite bouquets of circles are excluded.
In particular, there are not connected attractors of IFS of contractive similarities in the plane having empty interior and with the shape of the circle.
This is in sharp contrast with the situation for flows or homeomorphisms.
We have the following general result in this direction.

\begin{theorem}
Let $\mathcal{F}$ be an iterated function system of $\ \mathbb{R}^{2}$
consisting of contractive homeomorphisms and suppose that the attractor $K$ of $\mathcal{F}$ is a continuum with empty interior. Then $K$ has the shape of a point or the shape of the Hawaiian earring.
As a consequence, $\check{H} _{1}(K)$ (with coefficients in $\mathbb{Z}$) is either trivial or isomorphic to $\bigoplus_{n=1}^{\infty } \mathbb{Z}$.
\end{theorem}

\begin{proof}
Suppose on the contrary that the continuum $K$ has the shape of a finite bouquet of circles and that there exists an IFS $\mathcal{F} =\{h_{1},\dots ,h_{r}\}$ consisting of contractive homeomorphisms of $\mathbb{R} ^{2}$ having $K$ as its limit. Then
$$K=h_{1}(K)\cup \dots \cup h_{r}(K).$$

We shall show that this impossible. As a matter of fact, we shall show that there is no contractive homeomorphism $h$ of $\mathbb{R}^{2}$ such that $h(K)\subset K$. We suppose provisionally the contrary and we will get a contradiction.

Let $A$ be a bounded component of $\mathbb{R}^{2}\setminus K$ (which exists because we are assuming that $K$ has not trivial shape). Then there exists $a\in A$ such that $h(a)\in B$, where $B$ is also a bounded component of $\mathbb{R} ^{2}\setminus K$. Otherwise $h(A)\subset C\cup K$ where $C$ is the unbounded component of $\mathbb{R}^{2}\setminus K$. Moreover, $h(A)$ is bounded since it is contained in the compact set $h(\bar{A})$.
Since $\mathring{K}$ is empty we have that $h(A)\cap C\neq \emptyset $.
Select a point $b\in h(A)\cap C$.
We know that $K$ has a system of neighborhoods in $\mathbb{R}^{2}$ consisting of (topological) disks with holes and, thus, there exists a disk $D_{0}$ containing $K$ in its interior with $b\notin D_{0}$.
Since $h(A)$ is bounded there exists a bigger disk $D_{1}$ that contains $h(A)$ and such that $D_{0}\subset \mathring{D}_{1}$.
It is easy to see that in this situation there are points in $\partial h(A)$ contained in $D_{1}\setminus \mathring{D}_{0}$ and, hence, these points are not in $K$. However $\partial h(A)=h(\partial
A)\subset h(K)\subset K$ and this is a contradiction.

Now select a bounded component $A$ of $\mathbb{R}^{2}\setminus K$ such that $l=\diam (\partial A)=\diam(\bar{A})$ is minimal among all the bounded
components. The component $A$ and the number $l$ are well defined because $\mathbb{R}^{2}\setminus K$ has a finite number of components by our hypothesis about
the shape of $K$. We have seen that there exists $a\in A$ such that $h(a)\in
B$ (also bounded). Let $m=\diam(\partial B)$ (hence $l\leq m$). If $k<1$ is
the contractivity factor of $h$ we have that $kl<m$. There exists a disc $D'\subset B$ such that $h(a)\in D'$ and $\diam D'>kl$. Moreover $h(\partial A)\subset K\subset \mathbb{R}^{2}\setminus D'$. Select a disk $D\subset A$ containing $a$ and such that $\partial D$ is
near $\partial A$ so that $h(\partial D)\subset \mathbb{R}^{2}\setminus D'$.
We have then the following situation: 1) $h(a)\in D'$ and $h(a)\in
h(D)$ (which is also a disk), 2) $h(\partial D)\subset \mathbb{R}^{2}\setminus D'$. Hence $D'\subset h(D)$ and $\diam D'\leq \diam(h(D))$. However $\diam D'>kl$ and $\diam h(D)\leq kl$. This
contradiction establishes our result.
\end{proof}

\begin{corollary}
Let $\mathcal{F}$ be an iterated function system of $\mathbb{R}^2 $ consisting of contractive homeomorphisms and suppose that the attractor $K$ of $\mathcal{F}$ is a continuum with Hausdorff dimension less than 2. Then $K$ has the shape of a point or the shape of the Hawaiian earring.
\end{corollary}

We finish with a few remarks concerning Conley attractors for IFS, a theory
recently developed by Barnsley and Vince \cite{bav}.
The following remarks are presented without much detail and they are intended to suggest a possible line for future research.
Barnsley and Vince considered invertible systems $\mathcal{F}$ consisting of homeomorphisms of a compact metric space $X$ with no contractivity requirements.
A compact set $A\subset X$ is said to be a Conley attractor of $\mathcal{F}$ if there is an open set $U$ such that $A\subset U$ and $A=\lim_{k\rightarrow \infty } \mathcal{F}^{k}(\bar{U})$, where the limit is taken in the Hausdorff metric. They prove that such attractors are characterized by the existence of arbitrarily small attractor blocks, i.e. neighborhoods $Q$ of $A$ such that $\mathcal{F}(\bar{Q})\subset \mathring{Q}$ and
\[ A=\bigcap _{k\rightarrow \infty }\mathcal{F}^{k}(\bar{Q}). \]
An important ingredient in the Conley theory of flows is the notion of
continuation of an isolated invariant set. By using Barnsley and Vince
theorem on the existence of arbitrarily small attractor blocks it is
possible to give a meaning to the notion of continuation of a Conley
attractor for IFS. The following result can be easily proved.

\begin{theorem}
Let $\mathcal{F}_{\lambda }$, with $\lambda \in [0,1]$, be a
parametrized family (depending continuously on $\lambda $) of invertible IFS
consisting of homeomorphisms of a compact metric space $X$ and let $K_{0}$
be a Conley attractor for $\mathcal{F}_{0}$. Then there is a $\lambda _{0}$
such that for every $\lambda \leq \lambda _{0}$ there exists a Conley
attractor $K_{\lambda }$ of $\mathcal{F}_{\lambda }$ and $K_{\lambda
}\rightarrow K_{0}$ (i.e. all $K_{\lambda }$ are contained in an arbitrary
neighborhood of $K_{0}$ in $X$ for $\lambda $ sufficiently small).
\end{theorem}

The convergence is weaker than in the Hausdorff metric and the Conley
attractors $K_{\lambda }$ are uniquely determined, in the sense that they
are the only ones sharing the isolating block $Q$ with $K_{0}$. If we
require contractivity conditions in the basin of $K_{0}$ we can achieve
convergence in the Hausdorff metric. In this case we would get more general
versions of the "blowing in the wind" theorem \cite{bar}.
It would be interesting to study more general conditions implying convergence in the Hausdorff metric.

\begin{problem}
Study relations between the shape of the Conley attractor $K_{0}$ and the
shape of its continuations $K_{\lambda }$.
\end{problem}

\section*{Acknowledgements}

The authors are grateful to Jerzy Dydak for inspiring conversations. The first and third authors are supported by MINECO (MTM2015-63612-P).

\newpage
\centerline{\scshape H\'{e}ctor Barge}
\medskip
 {\footnotesize
 \centerline{E.T.S. Ingenieros inform\'{a}ticos}
   \centerline{Universidad Polit\'{e}cnica de Madrid}
   \centerline{28660 Madrid, Spain}
   \centerline{\email{h.barge@upm.es}}}

\medskip

\centerline{\scshape Antonio Giraldo}
\medskip
{\footnotesize
 \centerline{E.T.S. Ingenieros inform\'{a}ticos}
   \centerline{Universidad Polit\'{e}cnica de Madrid}
   \centerline{28660 Madrid, Spain}
   \centerline{\email{antonio.giraldo@upm.es}}}

\medskip

\centerline{\scshape Jos\'{e} M.R. Sanjurjo}
\medskip
{\footnotesize
 \centerline{Facultad de Ciencias Matem\'{a}ticas}
   \centerline{Universidad Complutense de Madrid}
   \centerline{28040 Madrid, Spain}
	\centerline{\email{jose\_sanjurjo@mat.ucm.es}}}

\end{document}